\documentclass{amsart}

\usepackage{amsmath}
\usepackage{amssymb}
\usepackage{xcolor}
\usepackage{url}

\newtheorem{thm}{Theorem}[section]

\newtheorem{lem}[thm]{Lemma}

\allowdisplaybreaks

\theoremstyle{definition}

\theoremstyle{remark}
\newtheorem{rem}{Remark}

\title[The convergence of discrete Fourier-Jacobi series]{The convergence of \\ discrete Fourier-Jacobi series}

\author[A. Arenas]{Alberto Arenas}
\address{Departamento de Matem\'aticas y Computaci\'on,
Universidad de La Rioja, Complejo Cient\'{\i}fico-Tecnol\'ogico,
Calle Madre de Dios 53, 26006 Logro\~no, Spain}
\email{alberto.arenas@unirioja.es}

\author[\'O. Ciaurri]{\'Oscar Ciaurri}
\address{Departamento de Matem\'aticas y Computaci\'on,
Universidad de La Rioja, Complejo Cient\'{\i}fico-Tecnol\'ogico,
Calle MAdre de Dios 53, 26006 Logro\~no, Spain}
\email{oscar.ciaurri@unirioja.es}

\author[E. Labarga]{Edgar Labarga}
\address{Departamento de Matem\'aticas y Computaci\'on,
Universidad de La Rioja, Complejo Cient\'{\i}fico-Tecnol\'ogico,
Calle Madre de Dios 53, 26006 Logro\~no, Spain}
\email{edgar.labarga@unirioja.es}

\keywords{Discrete harmonic analysis, $\ell^p(\mathbb{N})$-convergence, Jacobi polynomials, weighted norm inequalities}
\subjclass[2010]{Primary: 42C10.}
\thanks{The first-named author was supported by a predoctoral research grant of the Government of Comunidad Aut\'{o}noma de La Rioja. The second-named author was supported by grant PGC2018-096504-B-C32 from Spanish Government. The third-named author was supported by a predoctoral research grant of the University of La Rioja.}

\begin{document}

\begin{abstract}
The discrete counterpart of the problem related to the convergence of the Fourier-Jacobi series is studied.
To this end, given a sequence, we consider the analogue of the partial sum operator related to Jacobi polynomials and characterize its convergence in the $\ell^p(\mathbb{N})$-norm.
\end{abstract}

\maketitle

\section{Introduction}
By using Rodrigues' formula (see \cite[p.~67, eq.~(4.3.1)]{Szego}), the Jacobi polynomials $P^{(\alpha,\beta)}_n(x)$, $n\ge 0$, are defined as
\[
(1-x)^{\alpha}(1+x)^{\beta}P_n^{(\alpha,\beta)}(x)=\frac{(-1)^n}{2^n \, n!}\frac{d^n}{dx^n}\left((1-x)^{\alpha+n}(1+x)^{\beta+n}\right).
\]
For $\alpha,\beta>-1$, they are orthogonal on the interval $[-1,1]$ with respect to the measure
\[
d\mu_{\alpha,\beta}(x)=(1-x)^\alpha(1+x)^{\beta}\,dx.
\]
The family $\{p_n^{(\alpha,\beta)}(x)\}_{n\ge 0}$, given by $p_n^{(\alpha,\beta)}(x)=w_n^{(\alpha,\beta)}P_n^{(\alpha,\beta)}(x)$, where
\begin{equation*}
\begin{aligned}
w_n^{(\alpha,\beta)}& = \frac{1}{\|P_n^{(\alpha,\beta)}\|_{L^2((-1,1),d\mu_{\alpha,\beta})}} \\&= \sqrt{\frac{(2n+\alpha+\beta+1)\, n!\,\Gamma(n+\alpha+\beta+1)}{2^{\alpha+\beta+1}\Gamma(n+\alpha+1)\,\Gamma(n+\beta+1)}},\quad n\geq1,
\end{aligned}
\end{equation*}
and
\[
w_{0}^{(\alpha,\beta)} = \frac{1}{\|P_{0}^{(\alpha,\beta)}\|_{L^2((-1,1),d\mu_{\alpha,\beta})}} = \sqrt{\frac{\Gamma(\alpha+\beta+2)}{2^{\alpha+\beta+1}\Gamma(\alpha+1)\Gamma(\beta+1)}},
\]
is a complete orthonormal system in the space $L^2([-1,1],d\mu_{\alpha,\beta})$. Given a function $f\in L^2([-1,1],d\mu_{\alpha,\beta})$  its Fourier-Jacobi coefficients are defined by
\[
c_n^{(\alpha,\beta)}(f)=\int_{-1}^{1}f(x)p_n^{(\alpha,\beta)}(x)\, d\mu_{\alpha,\beta}(x).
\]
The application
\[
\begin{matrix}
L^2([-1,1],d\mu_{\alpha,\beta})&\longrightarrow &\ell^2(\mathbb{N})\\
f & \longmapsto & \{c_n^{(\alpha,\beta)}(f)\}_{n\ge 0}
\end{matrix}
\]
is an isometry and Parseval's identity
\[
\|f\|_{L^2([-1,1],d\mu_{\alpha,\beta})}=\|c_n^{(\alpha,\beta)}(f)\|_{\ell^2(\mathbb{N})}
\]
holds. For functions $f\in L^p([-1,1],d\mu_{\alpha,\beta})$, we define the $n$-th partial sum operator by
\[
S_n^{(\alpha,\beta)}f(x)=\sum_{k=0}^{n}c_k^{(\alpha,\beta)}(f)p_k^{(\alpha,\beta)}(x).
\]
It is well known (see \cite{Pollard} and \cite{New-Rudin}) that the mean convergence of $S_{n}^{(\alpha,\beta)}$, i.e.,
\begin{equation}
\label{eq:Sn-Jac}
S_n^{(\alpha,\beta)}f\longrightarrow f \qquad \text{in $L^p([-1,1],d\mu_{\alpha,\beta})$},
\end{equation}
holds for $\alpha,\beta\ge -1/2$ if and only if
\[
\max\left\{\frac{4(\alpha+1)}{2\alpha+3},\frac{4(\beta+1)}{2\beta+3}\right\}<p<\min\left\{\frac{4(\alpha+1)}{2\alpha+1},\frac{4(\beta+1)}{2\beta+1}\right\}.
\]
This partial sum operator has been extensively analysed. In \cite{Muck-Jac} some weighted inequalities were studied for $\alpha,\beta>-1$. The weak behaviour of $S_n^{(\alpha,\beta)}$ (weak $(p,p)$-type and restricted weak $(p,p)$-type inequalities) was treated in \cite{Chanillo} for the case $\alpha=\beta=0$ and in \cite{GPV-1} for the general case. Weighted weak type inequalities were analysed in \cite{GPV-2}.

In this paper, we focus on the analysis of discrete Fourier-Jacobi expansions. More precisely, given an appropriate sequence $\{f(n)\}_{n\ge 0}$, its $(\alpha,\beta)$-transform $\mathcal{F}_{\alpha,\beta}$ is given by the identity
\[
\mathcal{F}_{\alpha,\beta}f(x)=\sum_{k=0}^{\infty}f(k)p_k^{(\alpha,\beta)}(x)
\]
and its inverse by
\[
\mathcal{F}^{-1}_{\alpha,\beta}F(n)=c_n^{(\alpha,\beta)}(F).
\]
We are interested in recovering the given sequence by means of the multiplier of an interval for $\mathcal{F}_{\alpha,\beta}$. In a more concrete way, we define the multiplier of an interval $[a,b]\subset [-1,1]$, denoted by $T_{[a,b]}$ and simply by $\mathcal{T}_r$ when $[a,b]=[-r,r]$, with $0<r<1$, by the relation
\[
T_{[a,b]}f=\mathcal{F}^{-1}_{\alpha,\beta}(\chi_{[a,b]}\mathcal{F}_{\alpha,\beta}f).
\]
where $\chi_{[a,b]}$ is the characteristic function of the interval $[a,b]$. We want to study the conditions under
\begin{equation}
\label{eq:conv}
\lim_{r\to 1^{-}}\|\mathcal{T}_{r}f-f\|_{\ell^p(\mathbb{N})}=0.
\end{equation}
This problem is the discrete counterpart of \eqref{eq:Sn-Jac} and it belongs to the study of the discrete harmonic analysis for Jacobi series developed in \cite{ACL-JacI, ACL-JacII, ACL-Trans} by the authors. In those papers, the starting point is a discrete Laplacian defined by the three-term recurrence relation for the Jacobi polynomials. Recently, some classical operators in harmonic analysis have been treated in other discrete settings. For example, in \cite{Ciau-et-al} a complete study of the operators associated with the discrete Laplacian
\[
\Delta_d u(u)=u(n+1)-2u(n)+u(n-1), \qquad n\in \mathbb{Z},
\]
was carried out. On its behalf, the same analysis was done in \cite{Bet-et-al}  for a discrete Laplacian defined in terms of the three-term recurrence relation for the ultraspherical polynomials.

In order to study \eqref{eq:conv}, we give a complete characterization of the uniform boundedness of the operator $T_{[a,b]}$ on the spaces $\ell^p(\mathbb{N})$. This result will be a consequence of a more general one about the boundedness with discrete weights of $T_{[a,b]}$. Therefore, the convergence in \eqref{eq:conv} will follow from this characterization.

To state our result containing the weighted inequalities for the operator $T_{[a,b]}$, we need some preliminaries. A weight on $\mathbb{N}=\{0,1,2,\dots\}$ will be a strictly positive sequence $w=\{w(n)\}_{n\ge 0}$. We consider the weighted $\ell^{p}$-spaces
\[
\ell^p(\mathbb{N},w)=\left\{f=\{f(n)\}_{n\ge 0}: \|f\|_{\ell^{p}(\mathbb{N},w)}:=\Bigg(\sum_{m=0}^{\infty}|f(m)|^p w(m)\Bigg)^{1/p}<\infty\right\},
\]
$1\le p<\infty$, and the weak weighted $\ell^{1}$-space
\[
\ell^{1,\infty}(\mathbb{N},w)=\left\{f=\{f(n)\}_{n\ge 0}: \|f\|_{\ell^{1,\infty}(\mathbb{N},w)}:=\sup_{t>0}t\sum_{\{m\in \mathbb{N}: |f(m)|>t\}} w(m)<\infty\right\},
\]
and we simply write $\ell^p(\mathbb{N})$ and $\ell^{1,\infty}(\mathbb{N})$ when $w(n)=1$ for all $n\in \mathbb{N}$.

Furthermore, we say that a weight $w(n)$ belongs to the discrete Muckenhoupt $A_p(\mathbb{N})$ (see, for instance, \cite{HMW}) if
\[
[w]_{A_p(\mathbb{N}):=}\sup_{\begin{smallmatrix} 0\le n \le m \\ n,m\in \mathbb{N} \end{smallmatrix}} \frac{1}{(m-n+1)^p}\Bigg(\sum_{k=n}^mw(k)\Bigg)\Bigg(\sum_{k=n}^mw(k)^{-1/(p-1)}\Bigg)^{p-1} <\infty,
\]
for $1<p<\infty$,
\[
[w]_{A_1(\mathbb{N})}:=\sup_{\begin{smallmatrix} 0\le n \le m \\ n,m\in \mathbb{N} \end{smallmatrix}} \frac{1}{m-n+1}\Bigg(\sum_{k=n}^mw(k)\Bigg)\max_{n\le k \le m}w(k)^{-1} <\infty,
\]
for $p=1$. The value $[w]_{A_p(\mathbb{N})}$ is called the $A_p(\mathbb{N})$ constant of $w$.

Now we are in position to state the following result.
\begin{thm}
\label{thm:main}
Let $\alpha,\beta \ge -1/2$, $[a,b]\subset [-1,1]$, $1\le p<\infty$, and $w\in A_p(\mathbb{N})$. Then, 
\[
T_{[a,b]}f(n)=\sum_{m=0}^{\infty}f(m)K_{[a,b]}(m,n), \qquad f\in \ell^p(\mathbb{N},w),
\]
where
\[
K_{[a,b]}(m,n)=\int_{a}^{b}p_m^{(\alpha,\beta)}(x)p_n^{(\alpha,\beta)}(x)\, d\mu_{\alpha,\beta}(x).
\]
Moreover, for $1<p<\infty$
\begin{equation}
\label{eq:main-estimate}
\|T_{[a,b]}f\|_{\ell^p(\mathbb{N},w)}\le C \|f\|_{\ell^p(\mathbb{N},w)},
\end{equation}
and for $p=1$
\begin{equation}
\label{eq:main-estimate-1}
\|T_{[a,b]}f\|_{\ell^{1,\infty}(\mathbb{N},w)}\le C \|f\|_{\ell^1(\mathbb{N},w)},
\end{equation}
where $C$ is a constant independent of $f$ and $[a,b]$ in both inequalities.
\end{thm}

As a consequence of the previous theorem, we can characterize the uniform boundedness of $T_{[a,b]}$ on the spaces $\ell^p(\mathbb{N})$.

\begin{thm}
\label{thm:main-nopesos}
Let $\alpha,\beta \ge -1/2$ and $1\le p<\infty$. Then,
\begin{equation}
\label{eq:main-estimate-nopesos}
\|T_{[a,b]}f\|_{\ell^p(\mathbb{N})}\le C \|f\|_{\ell^p(\mathbb{N})},
\end{equation}
where $C$ is a constant independent of $f$ and $[a,b]\subset [-1,1]$, if and only if $1<p<\infty$.
\end{thm}

Finally, from Theorem \ref{thm:main-nopesos}, we deduce that
\begin{thm}
\label{thm:conv}
Let $\alpha,\beta \ge -1/2$ and $1\le p<\infty$. Then \eqref{eq:conv} holds if and only if $1<p<\infty$.
\end{thm}

Of course, from Theorem \ref{thm:conv}, the pointwise convergence
\[
\lim_{r\to 1^{-}}\mathcal{T}_rf(n)=f(n), \qquad n\in\mathbb{N},
\]
follows immediately.

The paper is organised as follows. Section \ref{sec:main} contains the proof of Theorem \ref{thm:main}. To prove it we obtain a proper expression for the kernel of $T_{[a,b]}$ to write it in terms of some classical operators. The mapping properties of such operators will be used to complete the result. The proofs of Theorem \ref{thm:main-nopesos} and Theorem \ref{thm:conv} are contained in Section \ref{sec:proofs} where some technical lemmas are also included.
\section{Proof of Theorem \ref{thm:main}}
\label{sec:main}
From the identity
\[
\chi_{[a,b]}(x)=\chi_{[-1,b]}(x)-\chi_{[-1,a)}(x),
\]
we can focus on analysing the operator $T_{[-1,b]}$, denoted by $T_b$. For sequences $f\in\ell^2(\mathbb{N})\cap \ell^p(\mathbb{N},w)$, by using the identity
\[
\int_{-1}^{1}\mathcal{F}_{\alpha,\beta}g(x)\mathcal{F}_{\alpha,\beta}h(x)\, d\mu_{\alpha,\beta}(x)=\sum_{m=0}^{\infty}g(m)h(m),\qquad g,h\in \ell^2(\mathbb{N}),
\]
we have
\begin{equation}
\label{eq:Tb}
T_{b}f(n)=\sum_{m=0}^{\infty}f(m)K_b(m,n),
\end{equation}
where
\[
K_b(m,n)=\int_{-1}^{b}p_m^{(\alpha,\beta)}(x)p_n^{(\alpha,\beta)}(x)\, d\mu_{\alpha,\beta}(x).
\]
Our first step to prove Theorem \ref{thm:main} is to obtain an explicit expression for the kernel~$K_b$.

\begin{lem}
\label{lem:kernel}
Let $\alpha,\beta>-1$. Then, for $n\not=m$ we have the identity
\begin{equation*}
K_b(m,n)=\frac{(1-b)^{\alpha+1}(1+b)^{\beta+1}}{\lambda_n^{(\alpha,\beta)}-\lambda_m^{(\alpha,\beta)}}
\Big(p_n^{(\alpha,\beta)}(b)\big(p_m^{(\alpha,\beta)}\big)'(b)-\big(p_n^{(\alpha,\beta)}\big)'(b)p_m^{(\alpha,\beta)}(b)\Big),
\end{equation*}
where $\lambda_j^{(\alpha,\beta)}=j(j+\alpha+\beta+1)$.
\end{lem}
\begin{proof}
First, we note that (see \cite[p. 60, eq. (4.2.1)]{Szego})
\[
L^{\alpha,\beta}p_n^{(\alpha,\beta)}(x)=\lambda_n^{(\alpha,\beta)}p_n^{(\alpha,\beta)}(x),
\]
with
\[
L^{\alpha,\beta}=-(1-x^2)\frac{d^2}{dx^2}-(\beta-\alpha-(\alpha+\beta+2)x)\frac{d}{dx}.
\]
It is well known that $L^{\alpha,\beta}$ is a symmetric operator in $L^2([-1,1],d\mu_{\alpha,\beta})$, but for every interval $[r,s]\subset [-1,1]$, $r<s$, it is verified that
\begin{equation*}
\int_{r}^{s}f(x)L^{\alpha,\beta}g(x)\, d\mu_{\alpha,\beta}(x)=U_{\alpha,\beta}(f,g)(x)\Big|_{x=r}^{x=s}+\int_{r}^{s}g(x)L^{\alpha,\beta}f(x)\, d\mu_{\alpha,\beta}(x),
\end{equation*}
with
\[
U_{\alpha,\beta}(f,g)(x)=(1-x)^{\alpha+1}(1+x)^{\beta+1}\Big(g(x)\frac{df}{dx}(x)-f(x)\frac{dg}{dx}(x)\Big).
\]

Then,
\begin{multline*}
\lambda_n^{(\alpha,\beta)}K_b(m,n)=\int_{-1}^{b}p_m^{(\alpha,\beta)}(x)L^{\alpha,\beta}p_n^{(\alpha,\beta)}(x)\, d\mu_{\alpha,\beta}(x)\\
\begin{aligned}
&=U_{\alpha,\beta}(p_m^{(\alpha,\beta)},p_n^{(\alpha,\beta)})(x)\Big|_{x=-1}^{x=b}
+\int_{-1}^{b}L^{\alpha,\beta}p_m^{(\alpha,\beta)}(x)p_n^{(\alpha,\beta)}(x)\, d\mu_{\alpha,\beta}(x)\\&=
U_{\alpha,\beta}(p_m^{(\alpha,\beta)},p_n^{(\alpha,\beta)})(x)\Big|_{x=-1}^{x=b}+\lambda_m^{(\alpha,\beta)}K_b(m,n).
\end{aligned}
\end{multline*}
and
\[
K_{b}(m,n)=\frac{1}{\lambda_n^{(\alpha,\beta)}-\lambda_m^{(\alpha,\beta)}}
\left(U_{\alpha,\beta}(p_m^{(\alpha,\beta)},p_n^{(\alpha,\beta)})(x)\Big|_{x=-1}^{x=b}\right).
\]
Now the result follows immediately.
\end{proof}

The proof of Theorem \ref{thm:main} will be obtained by using the mapping properties of some classical operators. We consider
\[
Hf(n)=\sum_{\begin{smallmatrix}
              m=0 \\
              m\not=n
            \end{smallmatrix}}^{\infty}\frac{f(m)}{n-m}\qquad\text{ and }\qquad
Q_af(n)=\sum_{\begin{smallmatrix}
              m=0 \\
              m\not=n
            \end{smallmatrix}}^{\infty}\frac{f(m)}{n+m+a},
\]
for some non-negative constant $a$. In the definition of $Q_a$ we have considered $m\not=n$ because it is more convenient for us, but that value can be included without any problem.

The operator $H$ is the well known discrete Hilbert transform and its boundedness with weights was treated in \cite[Theorem 10]{HMW}. There, it was proved that
\begin{equation}
\label{eq:acot-P}
\|Hf\|_{\ell^p(\mathbb{N},w)}\le C \|f\|_{\ell^p(\mathbb{N},w)}\Longleftrightarrow w\in A_p(\mathbb{N}),
\end{equation}
for $1<p<\infty$, and
\begin{equation}
\label{eq:acotdeb-P}
\|Hf\|_{\ell^{1,\infty}(\mathbb{N},w)}\le C \|f\|_{\ell^1(\mathbb{N},w)} \Longleftrightarrow w\in A_1(\mathbb{N}).
\end{equation}
Moreover, the constant $C$ in \eqref{eq:acot-P} and \eqref{eq:acotdeb-P} only depends on the $A_p(\mathbb{N})$ constant of the weight $w$.

In the case of the operator $Q_a$, we have
\[
|Q_af(n)|\le C\left(\frac{1}{n+1}\sum_{m=0}^{n}|f(m)|+\sum_{m=n}^{\infty}\frac{|f(m)|}{m+1}\right)=:C\left(O_1f(n)+O_2f(n)\right).
\]
The operator $O_1$ is the discrete Hardy operator and it can be controlled by the discrete maximal operator, so it is bounded from $\ell^p(\mathbb{N},w)$ into itself when $1<p<\infty$ and $w\in A_p(\mathbb{N})$, and from $\ell^1(\mathbb{N},w)$ into $\ell^{1,\infty}(\mathbb{N},w)$ for $w\in A_1(\mathbb{N})$. 
From the identity
\[
\sum_{m=0}^{\infty}f(m)O_1g(m)=\sum_{m=0}^{\infty}g(m)O_2f(m),
\]
we have that $O_2$ is the adjoint operator of $O_1$ (in fact, it is the adjoint Hardy operator), and we conclude that
\begin{equation}
\label{eq:acot-Q}
\|Q_af\|_{\ell^p(\mathbb{N},w)}\le C \|f\|_{\ell^p(\mathbb{N},w)},
\end{equation}
for $1<p<\infty$ and $w\in A_p(\mathbb{N})$. Moreover, for $O_2$, using Fubini's theorem and the definition of $A_1(\mathbb{N})$, we can deduce that it is a bounded operator from $\ell^1(\mathbb{N},w)$ into itself and, finally, we have
\begin{equation}
\label{eq:acotdeb-Q}
\|Q_af\|_{\ell^{1,\infty}(\mathbb{N},w)}\le C \|f\|_{\ell^1(\mathbb{N},w)},
\end{equation}
when $w\in A_1(\mathbb{N})$. The constant appearing in the boundedness of the discrete maximal operator also  depends on the value $[w]_{A_p(\mathbb{N})}$ and, then, so it occurs for the constant $C$ in \eqref{eq:acot-Q} and \eqref{eq:acotdeb-Q}.

\begin{proof}[Proof of Theorem \ref{thm:main}]
Set
\[
r_b(n)=(1-b)^{\alpha/2+1/4}(1+b)^{\beta/2+1/4}p_n^{(\alpha,\beta)}(b)
\]
and
\[
R_b(n)=\frac{(1-b)^{\alpha/2+3/4}(1+b)^{\beta/2+3/4}}{2n+\alpha+\beta+1}(p_n^{(\alpha,\beta)})'(b).
\]
By Lemma \ref{lem:kernel} and the identities
\begin{align}
\label{eq:identities}
\frac{1}{\lambda_n^{(\alpha,\beta)}-\lambda_m^{(\alpha,\beta)}}&=\frac{1}{2m+\alpha+\beta+1}\left(\frac{1}{n-m}-\frac{1}{n+m+\alpha+\beta+1}\right)\\
&=\frac{1}{2n+\alpha+\beta+1}\left(\frac{1}{n-m}+\frac{1}{n+m+\alpha+\beta+1}\right)
\end{align}
we have
\begin{multline}
\label{eq:decom-T}
T_bf(n)=r_b(n)H(R_bf)(n)-R_b(n)H(r_bf)(n)-r_b(n)Q_{\alpha+\beta+1}(R_bf)(n)\\-R_b(n)Q_{\alpha+\beta+1}(r_bf)(n)+
f(n)K_b(n,n).
\end{multline}

To estimate the weights $r_b$ and $R_b$ we need some bounds for the Jacobi polynomials. For $a,b>-1$, the estimate (see \cite[eq.~(2.6) and (2.7)]{Muckenhoupt})
\begin{multline}
\label{eq:unif-bound-trozos}
  |p_n^{(a,b)}(x)|\\\le C \begin{cases}
                          (n+1)^{a+1/2}, & 1-1/(n+1)^{2}<x<1, \\
                          (1-x)^{-a/2-1/4}(1+x)^{-b/2-1/4}, & -1+1/(n+1)^{2}\leq x\leq 1-1/(n+1)^{2},\\
                          (n+1)^{b+1/2}, & -1<x<-1+1/(n+1)^{2},
                        \end{cases}
\end{multline}
holds, where $C$ is a constant independent of $n$ and $x$. When $a,b\ge -1/2$ the previous bound can be replaced by the simpler one
\begin{equation}
\label{eq:unif-bound}
|p_n^{(a,b)}(x)|\le C (1-x)^{-a/2-1/4}(1+x)^{-b/2-1/4}.
\end{equation}
In this way, using the identity (see \cite[eq. 18.9.15]{NIST})
\begin{equation*}
\frac{d P_n^{(a,b)}}{dx}(x)=\frac{n+a+b+1}{2}P_{n-1}^{(a+1,b+1)}(x)
\end{equation*}
and \eqref{eq:unif-bound},
we obtain the bounds
\begin{equation}
\label{eq:bound-weights}
|r_b(n)|\le C \qquad \text{ and }\qquad |R_b(n)|\le C.
\end{equation}

Then, by \eqref{eq:decom-T}, \eqref{eq:bound-weights}, \eqref{eq:acot-P}, \eqref{eq:acot-Q} and the estimate $K_b(n,n)\le 1$, we deduce that
\begin{multline*}
\|T_bf\|_{\ell^p(\mathbb{N},w)}\le C\left(\|H(R_bf)\|_{\ell^p(\mathbb{N},w)}+\|H(r_bf)\|_{\ell^p(\mathbb{N},w)}+\|Q_{\alpha+\beta+1}(R_bf)\|_{\ell^p(\mathbb{N},w)}
\right.\\ \left.
+\|Q_{\alpha+\beta+1}(r_bf)\|_{\ell^p(\mathbb{N},w)}+\|f\|_{\ell^p(\mathbb{N},w)}\right)\le C \|f\|_{\ell^p(\mathbb{N},w)}
\end{multline*}
and the proof of \eqref{eq:main-estimate} is completed when $f\in \ell^2(\mathbb{N})\cap\ell^p(\mathbb{N},w)$. To prove \eqref{eq:main-estimate-1} we proceed in the same way but using \eqref{eq:acotdeb-P} and \eqref{eq:acotdeb-Q} instead of \eqref{eq:acot-P} and \eqref{eq:acot-Q}.

At this point, we know that the operator $T_b$, which is given by \eqref{eq:Tb} for sequences $f\in \ell^2(\mathbb{N})\cap\ell^p(\mathbb{N},w)$, admits an extension, that we denote by $\mathbb{T}_b$, bounded from $\ell^p(\mathbb{N},w)$ into itself when $1<p<\infty$, and from $\ell^1(\mathbb{N},w)$ into $\ell^{1,\infty}(\mathbb{N},w)$. Let us see that
\[
 \mathbb{T}_bf(n)=\sum_{m=0}^{\infty}f(m)K_b^{(\alpha,\beta)}(m,n), \qquad f\in \ell^p(\mathbb{N},w),
\]
to complete the proof of our result. We provide the details for $1<p<\infty$ and we omit them for $p=1$ (see \cite{Bet-et-al}).

First, let us consider the functional
\begin{align*}
\mathfrak{T}_{b,n}: \ell^p(\mathbb{N},w)&\longrightarrow \mathbb{R}\\
f&\longmapsto \mathfrak{T}_{b,n}f:=\sum_{m=0}^{\infty}f(m)K_b^{(\alpha,\beta)}(m,n).
\end{align*}
For $1<p<\infty$, it is easy to check that
\begin{equation}
\label{eq:bound-Tbn}
|\mathfrak{T}_{b,n}f|\le C \frac{\|f\|_{\ell^p(\mathbb{N},w)}}{w^{1/p}(n)},
\end{equation}
for sequences $f\in \ell^2(\mathbb{N})\cap \ell^p(\mathbb{N},w)$. Then we can prove that $K_b(\cdot,n)$ is a sequence in $\ell^q(\mathbb{N},w^{-1/(p-1)})$, where $q$ is the conjugate exponent of $p$; i.e., $p^{-1}+q^{-1}=1$. 
Then, the operator $\mathfrak{T}_{b,n}^{(\alpha,\beta)}$ is bounded and it verifies \eqref{eq:bound-Tbn} for $f\in \ell^p(\mathbb{N},w)$ with $1< p<\infty$.

Now, given $f\in \ell^p(\mathbb{N},w)$ and  $\{f_k\}_{k\ge 0}\subset \ell^2(\mathbb{N})\cap \ell^p(\mathbb{N},w)$ such that $f_k\longrightarrow f$ in $\ell^p(\mathbb{N},w)$, we have
\[
\mathbb{T}_b f_k=T_bf_k \longrightarrow \mathbb{T}_bf, \qquad \text{ in } \ell^p(\mathbb{N},w)
\]
and
\[
\mathcal{T}_bf_k(n)=\mathfrak{T}_{b,n}f_k \longrightarrow \mathbb{T}_bf(n), \qquad \text{ in } \mathbb{R}.
\]
In this way, by the boundedness of $\mathfrak{T}_{b,n}$, we conclude that
\[
\mathbb{T}_bf(n)=\mathfrak{T}_{b,n}f=\sum_{m=0}^{\infty}f(m)K_b^{(\alpha,\beta)}(m,n), \qquad n\in \mathbb{N},
\]
finishing the proof
\end{proof}

\begin{rem}
For the complete range $\alpha,\beta>-1$, it is also possible to obtain \eqref{eq:main-estimate} and \eqref{eq:main-estimate-1} for $T_b$ but with more involved conditions on the weight $w$ than the simple one $w\in A_p(\mathbb{N})$. Indeed, from \eqref{eq:unif-bound-trozos} it is clear that, for $a,b>-1$,
\[
|p_n^{(a,b)}(x)|\le C \left(1-x+\frac{1}{(n+1)^2}\right)^{-a/2-1/4}\left(1+x+\frac{1}{(n+1)^2}\right)^{-b/2-1/4}.
\]
Then, taking
\[
M^{r,s}_{b}(n)=\left(\frac{1-b}{1-b+\frac{1}{(n+1)^2}}\right)^{r}\left(\frac{1+b}{1+b+\frac{1}{(n+1)^2}}\right)^{s},
\]
we have
\[
|r_b(n)|\le C M_b^{\alpha/2+1/4,\beta/2+1/4}(n) \qquad\text{ and }\qquad
|R_b(n)|\le C M_b^{\alpha/2+3/4,\beta/2+3/4}(n).
\]
In this way, provided that 
\[
w(n)M_b^{p(\alpha/2+1/4),p(\beta/2+1/4)}(n) 
\qquad\text{ and }\qquad
w(n)M_b^{p(\alpha/2+3/4),p(\beta/2+3/4)}(n)
\]
are uniform weights in $A_p(\mathbb{N})$ (uniform in the sense that the $A_p(\mathbb{N})$ constant of such weights does not depend on $b$) and using that 
\[
M_b^{\alpha/2+1/4,\beta/2+1/4}(n)M_b^{\alpha/2+3/4,\beta/2+3/4}(n)=M_b^{\alpha+1,\beta+1}(n)\le C, 
\]
it is possible to prove \eqref{eq:main-estimate} and \eqref{eq:main-estimate-1}. This fact is so because the constants in the boundedness of the discrete Hilbert transform and the discrete maximal function in $\ell^p(\mathbb{N},w)$ only depend on the $A_p(\mathbb{N})$ constant of the weight $w$.
\end{rem}

\section{Proofs of Theorem \ref{thm:main-nopesos} and Theorem \ref{thm:conv}}
\label{sec:proofs}
The main tool to prove Theorem \ref{thm:main-nopesos} and Theorem \ref{thm:conv} is the following lemma in which we analyse if $\{K_{b}(m,n)\}_{n\ge 0}$ is an element of $\ell^p(\mathbb{N})$.

\begin{lem}
\label{lem:bounds}
  Let $\alpha,\beta \ge -1/2$ and $m\in \mathbb{N}$. Then
  \begin{equation}
  \label{eq:bound-K}
  |K_b(m,n)|\le C |m-n|^{-1}, \qquad n\not=m,
  \end{equation}
and $K_b(m,\cdot)\in \ell^p(\mathbb{N})$ for $1<p<\infty$. Moreover,
  \begin{equation}
  \label{eq:asym}
  \sum_{n=m+1}^{2m}\left|\int_{0}^{1-1/m^2}p_m^{(\alpha,\beta)}(x)p_n^{(\alpha,\beta)}(x)\, d\mu_{\alpha,\beta}(x)\right|\simeq \log m.
  \end{equation}
\end{lem}
\begin{proof}
From Lemma \ref{lem:kernel}, applying the identities in \eqref{eq:identities} and the bounds for $r_b$ and $R_b$ in \eqref{eq:bound-weights}, we have the estimate $|K_b(m,n)|\le C|m-n|^{-1}$ for $n\not=m$. This estimate it is enough to show that $K_b(m,\cdot)\in \ell^p(\mathbb{N})$ for $1<p<\infty$ (note that $K_b(m,m)\le 1$).

Denoting by $I(m,n)$ the integral appearing in \eqref{eq:asym}, to obtain the result it is enough to prove that
\begin{equation}
\label{eq:asym-2}
I(m,n)= A\left(\frac{1}{n-m}+\frac{1}{N}\log\left(\frac{N}{n-m}\right)+\frac{1}{M}\log\left(\frac{M}{n-m}\right)+O(M^{-1})\right)
\end{equation}
for $m+1\le n\le 2m$, with $A$ a positive constant, $N=n+(\alpha+\beta+1)/2$, and $M=m+(\alpha+\beta+1)/2$. To attain this, we consider the expansion (deduce from known asymptotics for Jacobi polynomials in \cite[formula (9)]{Askey-Jacobi})
\begin{multline}
\label{eq:Darboux}
2^{(\alpha+\beta+1)/2}p_n^{(\alpha,\beta)}(\cos \theta)=(\sin\theta/2)^{-(\alpha+1/2)}(\cos\theta/2)^{-(\beta+1/2)}\\\times\left(A\cos\left(N\theta-\phi_\alpha\right)
+A\frac{\sin\left(N\theta-\phi_\alpha\right)}{N \theta}+O(N^{-1})+O((N\theta)^{-2})\right),
\end{multline}
for $\delta/n<\theta\le \pi/2$, with $\delta>0$, and $\phi_{\alpha}=(2\alpha+1)\pi/4$. Then, using the change of variable $x=\cos \theta$, taking $B=\arccos(1-1/m^2)$ (observe that $B\simeq 1/m \simeq 1/n$ for $m+1\le n \le 2m$), and applying \eqref{eq:Darboux} for $p_n^{(\alpha,\beta)}$ and $p_m^{(\alpha,\beta)}$, we have
\[
I(m,n)=J_1(m,n)+J_2(m,n)+J_3(m,n)+O(M^{-1}),
\]
with
\[
J_1(m,n)=A\int_{1/m}^{\pi/2}\cos(M\theta-\phi_\alpha)\cos(N\theta-\phi_\alpha)\, d\theta,
\]
\[
J_2(m,n)=\frac{A}{N}\int_{1/m}^{\pi/2}\cos(M\theta-\phi_\alpha)\sin(N\theta-\phi_\alpha)\, \frac{d\theta}{\theta}
\]
and $J_3(m,n)=J_2(n,m)$. Following \cite{Askey-Jacobi} (see \cite{Askey-ultra} for some technical details), we obtain that
\[
J_1(m,n)=\frac{A}{n-m}+O(M^{-1}),
\]
\[
J_2(m,n)=\frac{A}{N}\log\left(\frac{N}{n-m}\right)+O(M^{-1})
\]
and the similar estimates for $J_3(m,n)$ by changing the roles of $m$ and $n$. Now, the proof of \eqref{eq:asym-2} is completed. In this way, \eqref{eq:asym} follows immediately because
\[
\sum_{n=m+1}^{2m}\left(\frac{1}{N}\log\left(\frac{N}{n-m}\right)+\frac{1}{M}\log\left(\frac{M}{n-m}\right)\right)\le C
\]
and
\[
\sum_{n=m+1}^{2m}\frac{1}{n-m}\simeq \log m.\qedhere
\]
\end{proof}

Now, let us proceed with the proofs of Theorem \ref{thm:main-nopesos} and Theorem \ref{thm:conv}.

\begin{proof}[Proof of Theorem \ref{thm:main-nopesos}]
By Theorem \ref{thm:main} (note that $w(n)=1$ is a weight in $A_p(\mathbb{N})$ for $1<p<\infty$), it is enough to show the existence of a sequence $f\in\ell^{1}(\mathbb{N})$ such that the inequality
\begin{equation}
\label{eq:estimate-1}
\|T_{[a,b]}f\|_{\ell^1(\mathbb{N})}\le C \|f\|_{\ell^1(\mathbb{N})}
\end{equation}
does not hold for some interval $[a,b]$.

In this way, we take $m\ge 1$ and consider the interval $[a,b]=[0,1-1/m^2]$ and the sequence $f_{m}(n)=\delta_{nm}$, where $\delta_{nm}$ denotes the usual Kronecker delta function. Note that
\[
\mathcal{F}_{\alpha,\beta}f_m(x)=p_m^{(\alpha,\beta)}(x).
\]
Now, if \eqref{eq:estimate-1} was be true, it would imply
\[
1=\|f_m\|_{\ell^{1}(\mathbb{N})}\geq\sum_{n=m+1}^{2m}\left|\int_{0}^{1-1/m^{2}}p_{n}^{(\alpha,\beta)}(x)p_{m}^{(\alpha,\beta)}(x)
\,d\mu_{\alpha,\beta}(x)\right|,
\]
but this inequality is not possible because the right hand side is greater than $C\log m$ by~\eqref{eq:asym}.
\end{proof}

\begin{rem}
\label{rem}
By using the identity $p_n^{(a,b)}(-z)=(-1)^{n}p_n^{(b,a)}(z)$, for $-1<z<1$, proceeding as in the proof of Lemma \ref{lem:kernel}, it is possible to prove that
\begin{multline*}
\int_{-1+1/m^2}^{1-1/m^2}p_m^{(\alpha,\beta)}(x)p_n^{(\alpha,\beta)}(x)\, d\mu_{\alpha,\beta}(x)\\=A\left(\frac{1}{n-m}+\frac{1}{N}\log\left(\frac{N}{n-m}\right)+\frac{1}{M}\log\left(\frac{M}{n-m}\right)+O(M^{-1})\right).
\end{multline*}
Then, in particular, we can deduce that the operators $\mathcal{T}_r$ are not bounded from $\ell^1(\mathbb{N})$ into itself.
\end{rem}

To prove Theorem \ref{thm:conv}, first we have to check the convergence of $\mathcal{T}_r$ for sequence in $c_{00}$, the space of sequences having a finite number of non-null terms, and this is done in the following lemma.

\begin{lem}
\label{lem:conv-c00}
Let $\alpha,\beta\ge -1/2$, $1<p<\infty$, and $f\in c_{00}$. Then
\[
\lim_{r\to 1^-}\|\mathcal{T}_rf-f\|_{\ell^p(\mathbb{N})}=0.
\]
\end{lem}
\begin{proof}
Since each $f\in c_{00}$ could be stated by a finite linear combination of $f_{m}(n)=\delta_{nm}$, we prove the result for the latter sequences. Then, using that
\[
f_m(n)=\int_{-1}^{1}p_m^{(\alpha,\beta)}(x)p_n^{(\alpha,\beta)}(x)\, d\mu_{\alpha,\beta}(x),
\]
we have
\begin{multline}
\label{eq:diff}
\mathcal{T}_rf_m(n)-f_m(n)
=-\int_{-1}^{-r}p_{m}^{(\alpha,\beta)}(x)p_n^{(\alpha,\beta)}(x)\, d\mu_{\alpha,\beta}(x) \\-\int_{r}^{1}p_m^{(\alpha,\beta)}(x)p_n^{(\alpha,\beta)}(x)\, d\mu_{\alpha,\beta}(x).
\end{multline}
Owing to the orthogonality of the Jacobi polynomials, for $m\not=n$, it is verified that
\[
\int_{r}^{1}p_m^{(\alpha,\beta)}(x)p_n^{(\alpha,\beta)}(x)\, d\mu_{\alpha,\beta}(x)=-\int_{-1}^{r}p_m^{(\alpha,\beta)}(x)p_n^{(\alpha,\beta)}(x)\, d\mu_{\alpha,\beta}(x)
\]
and, from \eqref{eq:diff}, we deduce that
\[
\mathcal{T}_rf_m(n)-f_m(n)=K_{r}(m,n)-K_{-r}(m,n),\qquad m\not=n.
\]
When $n=m$, applying \eqref{eq:unif-bound} in \eqref{eq:diff}, the estimate
\[
\left|\mathcal{T}_rf_m(m)-f_m(m)\right|\le C(1-r)^{1/2}
\]
is attained. Then
\begin{equation}
\label{eq:limit}
\lim_{r\to 1^-}\|\mathcal{T}_rf_m-f_m\|_{\ell^p(\mathbb{N})}^p=\lim_{r\to 1^-}\sum_{\begin{smallmatrix}
                                                                                  n=0 \\
                                                                                  n\not= m
                                                                                \end{smallmatrix}}^{\infty}\left|K_{-r}(m,n)-K_r(m,n)\right|^{p}.
\end{equation}
From \eqref{eq:bound-K}, we have $\left|K_{-r}(m,n)-K_r(m,n)\right|\le C |m-n|^{-1}$, when $m\not=n$. Then applying the dominated convergence theorem the result follows (note that $|m-n|^{-1}$, for $n\not=m$, is $p$-summable for $1<p<\infty$) from \eqref{eq:limit} because $\lim_{r\to 1^-}(K_{-r}(m,n)-K_r(m,n))=0$.
\end{proof}

\begin{proof}[Proof of Theorem \ref{thm:conv}]
To prove \eqref{eq:conv} for $1<p<\infty$ and sequences $f\in\ell^p(\mathbb{N})$, it is enough to approximate them by sequences in $c_{00}$ and use Theorem \ref{thm:main-nopesos} and Lemma \ref{lem:conv-c00}. Indeed, given $\varepsilon >0$, we consider a sequence $g\in c_{00}$ such that $\|f-g\|_{\ell^p(\mathbb{N})}< \varepsilon$, then, applying \eqref{eq:main-estimate-nopesos}, we have
\begin{align*}
\|\mathcal{T}_rf-f\|_{\ell^p(\mathbb{N})}&\le \|\mathcal{T}_rf-\mathcal{T}_rg\|_{\ell^p(\mathbb{N})}+\|\mathcal{T}_rg-g\|_{\ell^p(\mathbb{N})}+\|g-f\|_{\ell^p(\mathbb{N})}\\&\le C \|g-f\|_{\ell^p(\mathbb{N})}+\|\mathcal{T}_rg-g\|_{\ell^p(\mathbb{N})}\le C \varepsilon,
\end{align*}
where in the last step we have used Lemma \ref{lem:conv-c00}.

The convergence in $\ell^1(\mathbb{N})$ is not possible because in such case the uniform boundedness principle would imply the uniform boundedness of the $\mathcal{T}_{r}$ in $\ell^1(\mathbb{N})$ and that is impossible (see Remark \ref{rem}).
\end{proof}

\section*{Acknowledgments}
The authors would like to thank the referees for the careful reading of the paper. Their suggestions and comments have substantially improved the final version of~it.


\end{document}